\documentclass[11pt,margin=1in,lmargin=1.25in]{amsart}

\usepackage{amsfonts, amsmath, amssymb, color}
\usepackage{amsthm}
\usepackage{bbm}
\usepackage{booktabs}
\usepackage{float}
\usepackage{hhline}
\usepackage{lipsum}
\usepackage{lscape}
\usepackage{subfig}
\usepackage{textcomp}
\usepackage{tikz}
\usetikzlibrary{decorations.pathmorphing,shapes,arrows,positioning}
\usepackage{xcolor}
\usepackage{young}
\usepackage{youngtab}
\usepackage{ytableau}
\allowdisplaybreaks[4]
\theoremstyle{plain}
\newtheorem{thm}{Theorem}[section]
\newtheorem{lem}[thm]{Lemma}
\newtheorem{prop}[thm]{Proposition}
\newtheorem{cor}[thm]{Corollary}
\newtheorem{rem}[thm]{Remark}

\theoremstyle{definition}

\newtheorem{exmp}[thm]{Example}

\newcommand{\la}{\lambda}

\newcommand{\tabincell}[2]{\begin{tabular}{@{}#1@{}}#2\end{tabular}}  %
\numberwithin{equation}{section} \errorcontextlines=0
\newcommand{\GL}{\mathrm{GL}}

\begin{document}
\title{$Q$-Kostka polynomials and spin Green polynomials}
\author{Anguo Jiang}
\address{School of Mathematics, South China University of Technology,
Guangzhou, Guangdong 510640, China}
\email{jagmath@163.com}
\author{Naihuan Jing}
\address{Department of Mathematics, North Carolina State University, Raleigh, NC 27695, USA}
\email{jing@ncsu.edu}
\author{Ning Liu}
\address{School of Mathematics, South China University of Technology,
Guangzhou, Guangdong 510640, China}
\email{mathliu123@outlook.com}
\maketitle

\begin{abstract}
We study the $Q$-Kostka polynomials $L_{\lambda\mu}(t)$ by the vertex operator realization of the $Q$-Hall-Littlewood functions $G_{\lambda}(x;t)$  and derive new formulae for $L_{\lambda\mu}(t)$. In particular, we have established stability property for the Q-Kostka polynomials. We also introduce
spin Green polynomials $Y^{\lambda}_{\mu}(t)$ as both an analogue of the Green polynomials and deformation of the spin irreducible characters of $\mathfrak S_n$. Iterative formulas of the spin Green polynomials are given and some favorable properties parallel to the Green polynomials are obtained.
Tables of $Y^{\lambda}_{\mu}(t)$ are included for $n\leq7.$
\end{abstract}

\section{Introduction}
Kostka-Foulkes polynomials have important applications in algebraic combinatorics and representation theory of the symmetric group $\mathfrak S_n$ and the finite general linear group $\GL_n(\mathbb F_q)$ (see \cite{DLT} and \cite{K}). They are defined as the transition coefficients between the Hall-Littlewood symmetric functions and Schur symmetric functions:
\begin{align*}
s_{\lambda}(x)=\sum\limits_{\mu}K_{\lambda\mu}(t)P_{\mu}(x;t)
\end{align*}
where both $s_{\lambda}$ and $P_{\mu}(t)$ form bases of the ring $\varLambda_F$ of symmetric functions over the field $F=\mathbb{Q}(t)$ and $P_{\mu}(0)=s_{\mu}$.

In \cite{G}, Green introduced a family of polynomials later called Green's polynomials for determining the irreducible complex characters of the group $\GL_n(\mathbb F_q)$. Morris used Littlewood symmetric functions to derive the Morris iteration formula \cite{MO} in studying the Green polynomials. Recently two of us have used vertex operators to establish combinatorial formulas for the Green polynomials and several new formulas are found in this approach \cite{JL}.

The classical theory of irreducible characters of the symmetric group admits a remarkable spin generalization started in Schur's classic work \cite{S}, and most constructions for the symmetric group and symmetric functions admit highly nontrivial spin counterparts, such as Schur $Q$-functions, shifted tableaux and Robinson-Schensted-Knuth correspondence \cite{Sa, St}. Notably the Frobenius type character formula was generalized to nontrivial part of irreducible characters of the double covering groups of $\mathfrak S_n$ with Schur's Q-functions taking the role of the usual Schur functions.

Let $\Gamma$ be the subalgebra of $\Lambda$ generated over $\mathbb{C}(t)$ by odd power sums $p_1, p_3, p_5, \ldots$, where Schur's $Q$-functions $Q_{\mu}(x)$ form an orthogonal basis
when $\mu$ runs through strict partitions. $\Gamma$ can also be interpreted as a representation ring of the projective representations of the symmetric groups $\mathfrak S_n$.
As a subring of the symmetric functions, $\Gamma$ has been well-studied in \cite{M, St}. In particular, Schur's Q-functions appear as specialization of Hall-Littlewood symmetric functions at $t=-1$: $Q_{\mu}(x)=Q_{\mu}(x; -1)$, where
$Q_{\mu}(x; t)$ are the dual of the $P_{\mu}(x; t)$.

Tudose and Zabrocki \cite{TZ} constructed another basis $G_{\lambda}(x;t)$ indexed by strict partitions in
$\Gamma$,  which is analogous to that of the Hall-Littlewood functions in $\Lambda$
so that $G_{\lambda}(x;0)=Q_{\lambda}(x)$. 
Based on this new basis,
they introduced an analogous Kostka-Foulkes polynomials
$L_{\lambda\mu}(t)$ as the coefficient of $Q_{\lambda}(x)$ in $G_{\mu}(x;t)$. The Q-Kostka polynomials $L_{\lambda\mu}(t)$ have similar properties as their classical counterparts and are conjecturally positive.
Their similarity and simplicity to the usual Kostka polynomials are obvious, and yet the combinatorial interpretation 
has yet to be uncovered. 
In this paper, we would like to use their vertex operator realization to give some new iterative formulas. Moreover, we also would like to apply these polynomials to study Schur's projective characters of $\mathfrak S_n$ and find new iterative formulas.

To this end, we introduce spin analogue of the Green polynomials and study their interpretation as matrix coefficients of certain vertex operators. Through this we are able to demonstrate some of their favorable properties similar to the Green polynomials, which in turn facilitates finding new iterative formulas for the spin Green polynomials. As a by-product we then obtain new iterative formulas for the spin irreducible characters of the symmetric group $\mathfrak S_n$.

We remark that there is another version of spin analog of Hall-Littlewood functions $Q_{\mu}^{-}(x;t)$ introduced by Wan and Wang \cite{W} in connection with representation theory of the Hecke-Clifford algebras, which will be treated elsewhere separately \cite{JL3}.

The structure of the paper is as follows. Section 2 reviews the vertex operator realization of $Q$-Hall-Littlewood functions $G_{\lambda}(x;t)$ and recalls some properties of $G_{\lambda}(x;t).$ In Section 3, we
derive an iterative formulas to compute $L_{\lambda\mu}(t)$ (see Theorem \ref{t:Literative}). The spin Green polynomials are introduced and some parallel properties are presented (see Theorem \ref{t:spinG}) in Section 4. In particular, we derive iterative formulas for the projective characters of the symmetric group $\mathfrak S_n$. The paper ends with tables of the Q-analog of Green's polynomials for $3\leq n\leq 7$.

\section{Vertex operator realization of $G_{\lambda}(x;t)$}

In this section, we briefly review the terminology of symmetric functions \cite{M} and introduce the $Q$-Hall-Littlewood polynomials and $Q$-Kostka polynomials \cite{TZ}.

A {\it partition} is any (finite or infinite) sequence of non-negative integers in decreasing order: $\lambda_1\geq\lambda_2\geq\ldots\geq\lambda_i\geq\ldots$ with finitely many non-zero terms. The sum of the parts is the {\it weight} of $\lambda$.
 If $|\lambda|=n$, we say that $\lambda$ is a partition of $n$. A strict (resp. odd) partition $\lambda$ means that all the parts $\lambda_l$ are distinct (resp. odd). We denote the set of all strict (resp. odd) partitions with weight $n$ by $\mathcal{SP}_n$ (resp. $\mathcal{OP}_n$). The number of non-zero parts is the length $l(\lambda)$ of $\lambda$. The {\it Young diagram} of partition $\lambda$ is defined as the set of points $(i,j)\in \mathbb{Z}^{2}$ such that $1\leq j \leq\lambda_i$. Often we replace the nodes by squares. For each partition $\lambda$, we define
\begin{align}
n(\lambda)&=\sum\limits_{i\geq1}(i-1)\lambda_i.
\end{align}
The dominance order $\geq$ of partitions is defined by
\begin{align*}
\lambda\geq\mu \Leftrightarrow |\lambda|=|\mu| \quad \text{and}\quad \lambda_1+\cdots+\lambda_i\geq\mu_1+\cdots+\mu_i, \quad\forall i\geq1.
\end{align*}

For a partition $\lambda=(\lambda_1,\ldots,\lambda_l)$, we define
\begin{align}
\lambda^{\hat{i}}=(\lambda_1,\lambda_2,\ldots,\lambda_{i-1},\lambda_{i+1},\ldots,\lambda_l) \quad i=1,2,\cdots,l.
\end{align}

In this paper, we use $t$-integer $[n]_t=t^{n-1}+t^{n-2}+\cdots+t+1$. As usual $[n]_t! = [n]_t\cdots[1]_t$ and $[0]_t=1$.  The Gauss $t$-binomial symbol is
\begin{align*}
\left[\begin{matrix}n\\k\end{matrix}\right]_t=\frac{[n]_t!}{[k]_t![n-k]_t!}.
\end{align*}
Similarly, we introduce $(k>0)$
\begin{align*}
(k)_t=\frac{t^k-(-1)^k}{t+1}=t^{k-1}-t^{k-2}+\cdots+(-1)^{k-2}t+(-1)^{k-1}.
\end{align*}
We make the convention that $(0)_t=1$ and $(k)_t=0$ for $k<0.$ The relation between $[k]_t$ and $(k)_t$ can be easily obtained \cite{JL2}:
\begin{align*}
(k)_t+2\sum_{i=1}^{k-1}(i)_t=[k]_t.
\end{align*}


Let $\Gamma_{\mathbb{Q}}$ denote the subring of $\Lambda$ generated by the $p_r$, $r$ odd. Clearly $\Gamma$ is a graded ring
$\Gamma=\oplus_{n\geq0}\Gamma^{n}$, where $\Gamma^{n}$ is spanned by the $p_\lambda, \lambda\in\mathcal{OP}$ such that $|\lambda|=n$.
$\Gamma$ is equipped with the bilinear form $\langle\ , \ \rangle$ defined by
\begin{align}
\langle p_{\lambda}, p_{\mu}\rangle=2^{-l(\lambda)}\delta_{\lambda\mu}z_{\lambda}.
\end{align}
for any odd partitions $\lambda, \mu$, where $z_{\lambda}=\prod_{i\geq 1}i^{m_i(\lambda)}m_i(\lambda)!$ and $m_i(\lambda)$ is the multiplicity of $i$ in the partition $\lambda$.

It is known \cite{M} that Schur's $Q$-functions $Q_{\lambda},$ $\lambda$ strict, form an orthogonal $\mathbb{Z}$-basis of $\Gamma$
\begin{align}\label{e:form}
\langle Q_{\lambda}, Q_{\mu}\rangle=2^{l(\lambda)}\delta_{\lambda\mu}, \qquad \lambda, \mu\in\mathcal{SP}.
\end{align}

The multiplication operator $p_n:\Gamma \rightarrow \Gamma$ is of degree $n$. So, the adjoint operator of $p_n$ is the differential operator $p_n^* =\frac{n}{2}\frac{\partial}{\partial p_n}$ of degree $-n$. Note that * is $\mathbb Q(t)$-linear and anti-involutive and satisfies
\begin{equation}
\langle f_nu, v\rangle=\langle u, f_n^*v\rangle
\end{equation}
for any linear homogeneous operator $f_n$ and $u, v\in \Gamma$.

The vertex operator of Schur's $Q$-functions are introduced in \cite{J2} as follows.
\begin{align}
\label{e:schurQop}
Q(z)&=\mbox{exp} \left( \sum\limits_{n\geq 1, \text{odd}} \dfrac{2}{n}p_nz^{n} \right) \mbox{exp} \left( -\sum \limits_{n\geq 1, \text{odd}} \frac{\partial}{\partial p_n}z^{-n} \right)\\ \notag
&=\sum_{n\in\mathbb Z}Q_nz^{n}.
\end{align}
So the adjoint vertex operator $Q^*(z)=Q(-z^{-1})$, thus $Q^*_n=Q_{-n}$. In other words,
$\langle Q_nu, v\rangle=\langle u, Q_{-n}v\rangle$ for any $u, v\in\Gamma$ and $n\in\mathbb Z$.

It is known that for any strict partition $\lambda=(\lambda_1,\ldots,\lambda_l)$, we have $Q_{\lambda}=Q_{\lambda_1}Q_{\lambda_2}\cdots Q_{\lambda_l}.\mathbbm{1},$ where $\mathbbm{1}$ is the {\it vacuum vector}. The commutation relations for the $Q_m$ are as follows.
\begin{prop}\label{t:relation} The components of $Q(z)$ generate a Clifford algebra with the following relations:
\begin{align}\label{e:relation}
\{Q_m, Q_n\}&
=(-1)^n2\delta_{m,-n}.
\end{align}
Moreover,
\begin{align}
Q_{-n}.\mathbbm{1}=\delta_{n,0}, \quad   (n\geq0)
\end{align}
where $\{A, B\}:=AB+BA$ .
\end{prop}

We introduce the operator $q(z)$ and its adjoint $q^*(z)$:
\begin{align}
q(z)&=\mbox{exp} \left( \sum\limits_{n\geq 1, \text{odd}} \dfrac{2}{n}p_nz^{n} \right) =\sum\limits_{n\geq0}q_nz^n,\\
q^*(z)&=\mbox{exp} \left(\sum \limits_{n\geq 1, \text{odd}} \frac{\partial}{\partial p_n}z^{-n} \right) =\sum\limits_{n\geq0}q^*_nz^{-n}.
\end{align}
Clearly, as a polynomial in the $p_k$, the polynomial $q_n$ is Schur's $Q$-function $Q_{(n)}$ associated with the one-row partition $(n)$. We define
\begin{align}
q_{\lambda}=q_{\lambda_1}q_{\lambda_2}\cdots q_{\lambda_{l}}.
\end{align}
Then $q_{\lambda}$, $\lambda$ strict, form a $\mathbb{Z}$-basis of $\Gamma$ .

In \cite{TZ}, Tudose and Zabrocki defined the $Q$-Hall-Littlewood function $G_{\lambda}(x;t)$
$\in \Gamma\otimes_{\mathbb{C}}\mathbb{C}(t)$:
\begin{align}
\label{e:defG}
G_{\lambda}(x;t):=\prod\limits_{i<j} \left(\dfrac{1+tR_{ij}}{1-tR_{ij}} \right) \left(\dfrac{1-R_{ij}}{1+R_{ij}} \right)q_{\lambda}=\prod\limits_{i<j} \left(\dfrac{1+tR_{ij}}{1-tR_{ij}} \right)Q_{\lambda}(x),
\end{align}
and the classical $Q$-function $Q_{\lambda}$ is the specialization at $t=0$.
Moreover, they have introduced the $q$-analog of the Kostka polynomials $L_{\lambda\mu}$ as the transition coefficients:
\begin{align}\label{e:defL}
G_{\mu}(x;t)=\sum\limits_{\lambda\in\mathcal{SP}}L_{\lambda\mu}(t)Q_{\lambda}(x), \qquad \mu\in\mathcal{SP}.
\end{align}

We remark that the $G_{\lambda},$ $\lambda$ strict, form a $\mathbb{Z}$-basis for $\Gamma\otimes_{\mathbb{Z}}\mathbb{Z}(t)$, but not an orthogonal one. Indeed, we have $G_{(3,2)}=Q_{(3,2)}+2tQ_{(4,1)}+2t^2Q_{(5)},$ $G_{(4,1)}=Q_{(4,1)}+2tQ_{(5)},$ but $\langle G_{(3,2)}, G_{(4,1)} \rangle=8t+8t^3\neq 0$ .

The $L_{\lambda\mu}(t)$ are called the {\it $Q$-Kostka polynomials} and this family of polynomials are known to share many properties with the classical Kostka-Foulkes polynomials. We will obtain more favorable properties parallel to the Kostka-Foulkes polynomials via vertex operators.

For this purpose, we introduce the vertex operators $G(z)$ and the adjoint operators $G^*(z)$
as the following  maps: $\Gamma\longrightarrow \Gamma[[z, z^{-1}]]=\Gamma\otimes\mathbb C(t)[[z, z^{-1}]]$ defined by
\begin{align}
\label{e:Qhallop}
G(z)&=\mbox{exp} \left( \sum\limits_{n\geq 1, \text{odd}} \dfrac{2}{n}p_nz^{n} \right) \mbox{exp} \left( \sum \limits_{n\geq 1, \text{odd}} (t^{n}-1) \frac{\partial}{\partial p_n}z^{-n} \right)\\ \notag
&=\sum_{n\in\mathbb Z}G_nz^{n},\\\label{e:Qhallop*}
G^*(z)&=\mbox{exp} \left(\sum\limits_{n\geq 1, \text{odd}} \dfrac{2(t^{n}-1)}{n}p_nz^{n} \right) \mbox{exp} \left(\sum \limits_{n\geq 1, \text{odd}} \frac{\partial}{\partial p_n}z^{-n} \right)\\ \notag
&=\sum_{n\in\mathbb Z}G^*_nz^{-n}.
\end{align}
Clearly, the components $G_n$ and $G_{-n}^*$ are annihilation operators of degree $-n$ for $n\geq0$. These vertex operators are analogues of the Hall-Littlewood vertex operator \cite{J1}. If we denote $G_{\lambda}.\mathbbm{1}=G_{\lambda_1}G_{\lambda_2}\cdots G_{\lambda_l}.\mathbbm{1}$, then clearly
they are equal to $G_{\lambda}$ by the usual vertex operator properties. Note that unlike the vertex operators for the Schur $Q$-functions, $G_nG_n.\mathbbm{1}~~(n>0)$ may not be zero! For example, $G_1G_1.\mathbbm{1}=2tQ_2\neq0.$

For $n\geq0,$ we have
\begin{align}\label{e:G*n}
G^*_{-n}.\mathbbm{1}&=\sum\limits_{\rho\in\mathcal{OP}_n}\frac{(-2)^{l(\rho)}}{z_{\rho}(t)}p_{\rho}.
\end{align}

By \eqref{e:form} and \eqref{e:defL}, we have
\begin{align*}
L_{\lambda\mu}(t)=2^{-l(\lambda)}\langle G_{\mu}(x;t), Q_{\lambda}(x) \rangle=2^{-l(\lambda)}\langle G_{\mu}.\mathbbm{1}, Q_{\lambda}.\mathbbm{1} \rangle.
\end{align*}

The vertex operators $G_n$ satisfy some quadratic relations. A $\lambda$-ring argument of the relations for $G_n$ were given in \cite{TZ}. In the following we use vertex operator technique to give a simpler proof.
\begin{prop}\label{t:GG}
We have the following commutation relations:
\begin{align}\label{e:GG}
\begin{split}
&(1-t^2)(G_mG_n+G_nG_m)+t(G_{m-1}G_{n+1}-G_{n+1}G_{m-1}\\
&+G_{n-1}G_{m+1}-G_{m+1}G_{n-1})=2(-1)^n(1-t)^2\delta_{m,-n}.
\end{split}
\end{align}
\end{prop}
\begin{proof}
By \eqref{e:Qhallop}, we obtain that
\begin{align*}
\dfrac{z-tw}{z+tw}G(z)G(w)+\dfrac{w-tz}{w+tz}G(w)G(z)&=\colon G(z)G(w) \colon (\dfrac{z-w}{z+w}+\dfrac{w-z}{w+z})\\
&=-2\delta(-\dfrac{w}{z})
\end{align*}
where $\colon G(z)G(w) \colon$ is the normal ordering product and $\delta(x)=\sum x^n$.
Multiplying $(z+tw)(w+tz)$ we then have that
\begin{align*}
&(z-tw)(w+tz)G(z)G(w)+(w-tz)(z+tw)G(w)G(z)\\
&=-2(1-t)^2z^2\delta(-\dfrac{w}{z}),
\end{align*}
and the quadratic relations follow by taking the coefficient of $z^{m+1}w^{n+1}$.
\end{proof}

\section{Iterative formulae for $L_{\lambda\mu}(t)$}

In this section, we give an iterative formula for $L_{\lambda\mu}(t)$ on $\mu$ and derive some new properties of $L_{\lambda\mu}$ in Proposition \ref{t:stability} and examples (Example \ref{t:tworows}).

Using the techniques of vertex operators, we also have the following commutation relations:
\begin{align}\label{e:hallop1}
G^*(z)Q(w)+Q&(w)G^*(z)\frac{w+tz}{w-tz}=2q(tz)\delta(\frac{w}{z}),\\ \label{e:hallop2}
q^{*}(z)G(w)&=G(w)q^{*}(z)\frac{z+w}{z-w},\\\label{e:hallop2}
Q(z)q(w)&=q(w)Q(z)\frac{z-w}{z+w}.
\end{align}

Taking components, we get the following result.
\begin{prop}\label{p:rel1}
The commutation relations between the operator $G_m$ and Schur's $Q$-function operators are:
\begin{align}\label{e:rel1}
G^{*}_{m}Q_{n}&=Q_nG^{*}_m+t^{-1}G^{*}_{m-1}Q_{n-1}+t^{-1}Q_{n-1}G^*_{m-1}\nonumber \\
&-2t^{n-m-1}(1-t)q_{n-m},\\ \label{e:rel2}
q^{*}_{m}G_{n}&=G_{n}q^{*}_{m}+q^{*}_{m-1}G_{n-1}+G_{n-1}q^{*}_{m-1},\\\label{e:rel3}
Q_mq_n&=q_nQ_m-Q_{m+1}q_{n-1}-q_{n-1}Q_{m+1}.
\end{align}
\end{prop}

Next we derive the iterative formula. 
\begin{thm}\label{t:G*Q}
For a given strict partition $\lambda=(\lambda_{1},\ldots ,\lambda_{l})$ and positive integer $k$, then
\begin{align}\label{e:G*Q}
G_{k}^{*}Q_{\lambda}.\mathbbm{1}&=\sum\limits_{i=1}^{l} (-1)^{i-1}2t^{\lambda_i-k}q_{\lambda_i-k}Q_{\lambda^{\hat{i}}}.\mathbbm{1}.
\end{align}
\end{thm}
\begin{proof}
This is proved by induction on $k+|\lambda|$. The first step is immediate. Assume that \eqref{e:G*Q} holds for $<k+|\lambda|$ , then
\begin{align*}
&Q_{\lambda_1}G^*_{k}Q_{\lambda_2}\cdots Q_{\lambda_l}.\mathbbm{1}\\
=&Q_{\lambda_1}(2t^{\lambda_2-k}q_{\lambda_2-k}Q_{\lambda_3}\cdots Q_{\lambda_l}.\mathbbm{1}
-2t^{\lambda_3-k}q_{\lambda_3-k}Q_{\lambda_2}Q_{\lambda_4}\cdots Q_{\lambda_l}.\mathbbm{1}\\
+&\cdots+(-1)^{l}2t^{\lambda_l-k}q_{\lambda_l-k}Q_{\lambda_2}\cdots Q_{\lambda_{l-1}}.\mathbbm{1}) .
\end{align*}
Also we have
\begin{align*}
&G^*_{k-1}Q_{\lambda_1-1}Q_{\lambda_2}\cdots Q_{\lambda_l}.\mathbbm{1}\\
=&2t^{\lambda_1-k}q_{\lambda_1-k}Q_{\lambda_2}\cdots Q_{\lambda_l}.\mathbbm{1}-2t^{\lambda_2-k+1}q_{\lambda_2-k+1}Q_{\lambda_1-1}Q_{\lambda_3}\cdots Q_{\lambda_l}.\mathbbm{1}\\
+&\cdots+(-1)^{l+1}2t^{\lambda_l-k+1}q_{\lambda_l-k+1}Q_{\lambda_1-1}Q_{\lambda_2}\cdots Q_{\lambda_{l-1}}.\mathbbm{1}
\end{align*}
and
\begin{align*}
&Q_{\lambda_1-1}G^*_{k-1}Q_{\lambda_2}\cdots Q_{\lambda_l}.\mathbbm{1}\\
=&Q_{\lambda_1-1}(2t^{\lambda_2-k+1}q_{\lambda_2-k+1}Q_{\lambda_3}\cdots Q_{\lambda_l}.\mathbbm{1}
-2t^{\lambda_3-k+1}q_{\lambda_3-k+1}Q_{\lambda_2}Q_{\lambda_4}\cdots Q_{\lambda_l}.\mathbbm{1}\\
+&\cdots+(-1)^{l}2t^{\lambda_l-k+1}q_{\lambda_l-k+1}Q_{\lambda_2}\cdots Q_{\lambda_{l-1}}.\mathbbm{1})
\end{align*}
It follows from \eqref{e:rel1} that
\begin{align*}
&G^*_{k}Q_{\lambda_1}Q_{\lambda_2}\cdots Q_{\lambda_l}.\mathbbm{1}\\
=&Q_{\lambda_1}G^*_{k}Q_{\lambda_2}\cdots Q_{\lambda_l}.\mathbbm{1}+t^{-1}G^*_{k-1}Q_{\lambda_1-1}Q_{\lambda_2}\cdots Q_{\lambda_l}.\mathbbm{1}\\
+&t^{-1}Q_{\lambda_1-1}G^*_{k-1}Q_{\lambda_2}\cdots Q_{\lambda_l}.\mathbbm{1}
-2t^{\lambda_1-k-1}(1-t)q_{\lambda_1-k}Q_{\lambda_2}\cdots Q_{\lambda_l}.\mathbbm{1} .
\end{align*}
Expanding $G^*_{k}Q_{\lambda_2}.\mathbbm{1}$ , $G^*_{k-1}Q_{\lambda_1-1}.\mathbbm{1}$ and $Q_{\lambda_1-1}G^*_{k-1}.\mathbbm{1}$ by induction, we then obtain
\begin{align*}
&G^*_{k}Q_{\lambda_1}Q_{\lambda_2}\cdots Q_{\lambda_l}.\mathbbm{1}\\
=&Q_{\lambda_1}(2t^{\lambda_2-k}q_{\lambda_2-k}Q_{\lambda_3}\cdots Q_{\lambda_l}.\mathbbm{1}
-2t^{\lambda_3-k}q_{\lambda_3-k}Q_{\lambda_2}Q_{\lambda_4}\cdots Q_{\lambda_l}.\mathbbm{1}\\
+&\cdots+(-1)^{l}2t^{\lambda_l-k}q_{\lambda_l-k}Q_{\lambda_2}\cdots Q_{\lambda_{l-1}}.\mathbbm{1})\\
+&t^{-1}(2t^{\lambda_1-k}q_{\lambda_1-k}Q_{\lambda_2}\cdots
Q_{\lambda_l}.\mathbbm{1}-2t^{\lambda_2-k+1}q_{\lambda_2-k+1}Q_{\lambda_1-1}Q_{\lambda_3}\cdots Q_{\lambda_l}.\mathbbm{1}\\
+&\cdots+(-1)^{l+1}2t^{\lambda_l-k+1}q_{\lambda_l-k+1}Q_{\lambda_1-1}Q_{\lambda_2}\cdots Q_{\lambda_{l-1}}.\mathbbm{1})\\
+&t^{-1}Q_{\lambda_1-1}(2t^{\lambda_2-k+1}q_{\lambda_2-k+1}Q_{\lambda_3}\cdots Q_{\lambda_l}.\mathbbm{1}-2t^{\lambda_3-k+1}q_{\lambda_3-k+1}Q_{\lambda_2}Q_{\lambda_4}\cdots Q_{\lambda_l}.\mathbbm{1}\\
+&\cdots+(-1)^{l}2t^{\lambda_l-k+1}q_{\lambda_l-k+1}Q_{\lambda_2}\cdots Q_{\lambda_{l-1}}.\mathbbm{1})\\
-&2t^{\lambda_1-k-1}(1-t)q_{\lambda_1-k}Q_{\lambda_2}\cdots Q_{\lambda_l}.\mathbbm{1}
\end{align*}
Using the commutation relation \eqref{e:rel3}, we then compute that
\begin{align*}
&Q_{\lambda_1-1}(2t^{\lambda_2-k+1}q_{\lambda_2-k+1}Q_{\lambda_3}\cdots Q_{\lambda_l}.\mathbbm{1}\\
&\qquad\qquad +\cdots+(-1)^{l}2t^{\lambda_l-k+1}q_{\lambda_l-k+1}Q_{\lambda_2}\cdots Q_{\lambda_{l-1}}.\mathbbm{1})\\
=&t^{-1}(2t^{\lambda_2-k+1}q_{\lambda_2-k+1}Q_{\lambda_1-1}Q_{\lambda_3}\cdots Q_{\lambda_l}.\mathbbm{1}-2t^{\lambda_2-k+1}Q_{\lambda_1}q_{\lambda_2-k}Q_{\lambda_3}\cdots Q_{\lambda_l}.\mathbbm{1}\\
&-2t^{\lambda_2-k+1}q_{\lambda_2-k}Q_{\lambda_1}Q_{\lambda_3}\cdots Q_{\lambda_l}.\mathbbm{1})\\
&-t^{-1}(2t^{\lambda_3-k+1}q_{\lambda_3-k+1}Q_{\lambda_1-1}Q_{\lambda_2}Q_{\lambda_4}\cdots Q_{\lambda_l}.\mathbbm{1}\\
&-2t^{\lambda_3-k+1}Q_{\lambda_1}q_{\lambda_3-k}Q_{\lambda_2}Q_{\lambda_4}\cdots Q_{\lambda_l}.\mathbbm{1}\\
&-2t^{\lambda_3-k+1}q_{\lambda_3-k}Q_{\lambda_1}Q_{\lambda_2}Q_{\lambda_4}\cdots Q_{\lambda_l}.\mathbbm{1})\\
&+\cdots+(-1)^lt^{-1}(2t^{\lambda_l-k+1}q_{\lambda_l-k+1}Q_{\lambda_1-1}Q_{\lambda_2}\cdots Q_{\lambda_{l-1}}.\mathbbm{1}\\
&-2t^{\lambda_l-k+1}Q_{\lambda_1}q_{\lambda_l-k}Q_{\lambda_2}\cdots Q_{\lambda_{l-1}}.\mathbbm{1}\\
&-2t^{\lambda_l-k+1}q_{\lambda_l-k}Q_{\lambda_1}Q_{\lambda_2}\cdots Q_{\lambda_{l-1}}.\mathbbm{1})
\end{align*}
which can be simplified to
\begin{align*}
&G^*_{k}Q_{\lambda_1}Q_{\lambda_2}\cdots Q_{\lambda_l}.\mathbbm{1}\\
=&2t^{\lambda_1-k}q_{\lambda_1-k}Q_{\lambda_2}\cdots Q_{\lambda_l}.\mathbbm{1}-2t^{\lambda_2-k}q_{\lambda_2-k}Q_{\lambda_1}\cdots Q_{\lambda_l}.\mathbbm{1}\\
+&\cdots+2(-1)^{l-1}t^{\lambda_l-k}q_{\lambda_l-k}Q_{\lambda_2}\cdots Q_{\lambda_{l-1}}.\mathbbm{1}\\
=&\sum\limits_{i=1}^{l} (-1)^{i-1}2t^{\lambda_i-k}q_{\lambda_i-k}Q_{\lambda^{\hat{i}}}.\mathbbm{1}.
\end{align*}
\end{proof}

\begin{exmp}
For $\lambda=(4,1)$ and $\mu=(3,2)$ , we have
\begin{align*}
L_{\lambda\mu}(t)&=2^{-l(\lambda)}\langle G_3G_2.1, Q_4Q_1.1 \rangle\\
&=\frac{1}{4}\langle G_2.1, 2tq_1Q_1.1 \rangle\\
&=\frac{1}{4}\langle (G_2.1, 4tQ_2.1 \rangle\\
&=2t.
\end{align*}
\end{exmp}

Theorem \ref{t:G*Q} gives an algebraic iterative formula for $L_{\lambda\mu}(t).$ Let us recall the Pieri formula for Schur's $Q$-function. For two partitions $\mu\subset\lambda$, the set-theoretic difference $\theta=\lambda-\mu$ is called a {\it skew diagram} denoted by $\lambda/\mu$. A skew diagram $\lambda/\mu$ is a vertical (resp. horizontal) strip if each row (resp. column) contains at most one box (see \cite{M}). Let $\lambda, \mu$ be strict partitions such that $\mu\subset\lambda$, and $\lambda/\mu$ is a horizontal strip. Denote by $a(\lambda/\mu)$ the number of integers $i\geq1$ such that $\lambda/\mu$ has a square in the $i$th column but not in the $(i+1)$st column. Then we have
\begin{align}\label{e:pieri}
2^{-l(\mu)}Q_{\mu}q_{r}=\sum\limits_{\lambda}2^{a(\lambda/\mu)}2^{-l(\lambda)}Q_{\lambda}
\end{align}
where the sum runs over strict partitions $\lambda\supset\mu$ such that $\lambda/\mu$ are horizontal $r$-strips.

\begin{thm}\label{t:Literative}
For given strict partitions $\lambda, \mu$ of $n$ we have an algebraic iterative formula for $L_{\lambda\mu}(t)$:
\begin{align}\label{e:iterative}
L_{\lambda\mu}(t)=\sum\limits_{\lambda_i\geq \mu_1}\sum\limits_{\xi}(-1)^{i-1}2^{a(\xi/\lambda^{\hat{i}})}t^{\lambda_i-\mu_1}L_{\xi\mu^{\hat{1}}}(t)
\end{align}
where the second sum runs over strict partitions $\xi\supset\lambda^{\hat{i}}$ such that $\xi/\lambda^{\hat{i}}$ are horizontal $\mu_1$-strips.
\end{thm}
\begin{proof} The $L_{\lambda\mu}(t)$ can be expressed as matrix coefficients of vertex operators:
\begin{align*}
L_{\lambda\mu}(t)&=2^{-l(\lambda)}\langle G_{\mu}.\mathbbm{1}, Q_{\lambda}.\mathbbm{1} \rangle\\
&=2^{-l(\lambda)}\langle G_{\mu^{\hat{1}}}.\mathbbm{1}, G^*_{\mu_1}Q_{\lambda}.\mathbbm{1} \rangle.
\end{align*}
Then \eqref{e:iterative} follows from \eqref{e:G*Q}, \eqref{e:pieri} and the definition of $L_{\xi\mu^{\hat{1}}}(t)$.
\end{proof}

We remark that Tudose and Zabrocki obtained another iterative formula similar to Morris' iteration \cite{TZ}.
\begin{exmp}\label{t:tworows}
Let $\lambda, \mu=(\mu_1,\mu_2)$ be two strict partitions of $n$, then we have
\begin{align}\label{e:tworows}
L_{\lambda(\mu_1,\mu_2)}(t)&=
\begin{cases}
2^{\delta_{\lambda,\mu}}t^{\lambda_1-\mu_1}& \text{if $\lambda\geq(\mu_1,\mu_2)$}\\
0& \text{if others.}
\end{cases}
\end{align}
\end{exmp}

\begin{cor}
Let $\mu \leq \lambda$ in the dominance order. \\
(1) If $n>\lambda_{1}$, then $L_{(n ,\lambda)(n ,\mu)}(t)=L_{\lambda\mu}(t)$. \\
(2) $L_{\lambda\lambda}(t)=1 $ and $L_{(|\lambda|)\lambda}(t)=2^{l(\lambda)-1}t^{n(\lambda)}$. \\
(3) $2^{l(\mu)-l(\lambda)} \mid L_{\lambda\mu}(t)$. \\
(4) $degL_{\lambda\mu}(t)=n(\mu)-n(\lambda)$.
\end{cor}
\begin{proof}
(1) and (2) are straightforward consequences of \eqref{e:iterative}. (3) and (4) are proved by induction on $l(\mu)$ using \eqref{e:iterative}.
\end{proof}

Recall that the Kostka-Foulkes polynomials have the stability property (see \cite{BJ}) and we have the following Q-analog, which proves a conjecture in \cite[Conj. 10]{TZ}.
\begin{prop}\label{t:stability}
Let $\lambda=(\lambda_1,\cdots,\lambda_l), \mu=(\mu_1,\cdots,\mu_m)$ be two strict partitions and $\mu_1 \geq \lambda_2$, then for any $r\geq1$, we have the stability
\begin{align}
L_{\lambda+(r),\mu+(r)}(t)=L_{\lambda\mu}(t).
\end{align}
where $\lambda+(r)=(\lambda_1+r,\lambda_2,\cdots)$.
\end{prop}
\begin{proof}
By Theorem \ref{t:G*Q}, we can compute $L_{\lambda+(r),\mu+(r)}(t)$ and $L_{\lambda\mu}(t)$ respectively and they are both equal to
\begin{align*}
\langle G_{\mu_2}G_{\mu_3}\cdots G_{\mu_m}.\mathbbm{1}, 2t^{\mu_1-\lambda_1}q_{\mu_1-\lambda_1}Q_{\lambda_2}\cdots Q_{\lambda_l}.\mathbbm{1} \rangle.
\end{align*}
\end{proof}

\section{Spin Green polynomials}
In this section, we introduce Q-analog Green's polynomials and show that they
enjoy similar favorable properties parallel to Green's polynomials. As an application
we also derive iterative formulas for the spin characters of the symmetric group $\mathfrak S_n$.

Recall Green's polynomials form the transition matrix $X^{\lambda}_{\mu}(t)$ between the Hall-Littlewood functions $Q_{\lambda}(x;t)$ and the power-sum product $p_{\mu}(x)$:
\begin{align}\label{e:X}
Q_{\lambda}(x;t)=\sum_{\mu}z_{\mu}(t)^{-1}X^{\lambda}_{\mu}(t)p_{\mu}(x).
\end{align}

Historically Green's polynomials $Q^{\lambda}_{\mu}(q)$ are defined by
\begin{align*}
Q^{\lambda}_{\mu}(q)=q^{n(\lambda)}X^{\lambda}_{\mu}(q^{-1}).
\end{align*}
In particular, $X^{\lambda}_{\mu}(0)=\chi^{\lambda}_{\mu}$,
the value of the irreducible character $\chi^{\lambda}$ of the symmetric group $\mathfrak{S}_n$ at the elements of cycle-type $\mu$.

We introduce the spin Green polynomials $Y^{\lambda}_{\mu}(t)$ as the spin analog of \eqref{e:X}: the transition matrix between the $Q$-Hall-Littlewood functions and the power-sum product, i.e. ,
\begin{align}
G_{\lambda}(x;t)=\sum_{\mu\in\mathcal{OP}_n}z_{\mu}^{-1}2^{l(\mu)}Y^{\lambda}_{\mu}(t)p_{\mu}(x).
\end{align}

Let $\zeta^{\lambda}_{\mu}$ be the irreducible negative character of the double cover $\tilde{\mathfrak{S}}_n$, indexed by $\lambda\in\mathcal{SP}_n$ and evaluated at the conjugacy class $\mu\in\mathcal{OP}_n$. The $\zeta^{\lambda}_{\mu}$ can be determined by the following Frobenius type formula \cite{S, HH}:
\begin{align}
Q_{\lambda}(x)=\sum_{\mu\in\mathcal{OP}_n}z_{\mu}^{-1}2^{\frac{l(\lambda)+l(\mu)+\epsilon(\lambda)}{2}}\zeta^{\lambda}_{\mu}p_{\mu}(x),
\end{align}
where $\epsilon(\lambda)=
\begin{cases}
0&\text{if $n-l(\lambda)\equiv0$ (mod 2)},\\
1&\text{if $n-l(\lambda)\not\equiv0$ (mod 2)}.
\end{cases}$

Similar to the Green polynomials, the spin Green polynomials $Y^{\lambda}_{\mu}(t)$ can be viewed as $t$-deformation of the spin characters $\zeta^{\lambda}_{\mu}$ (up to powers of $2$) since $G_{\lambda}(x;0)=Q_{\lambda}(x)$. Indeed, we have $Y^{\lambda}_{\mu}(0)=2^{\frac{l(\lambda)-l(\mu)+\epsilon(\lambda)}{2}}\zeta^{\lambda}_{\mu}.$

We can write the spin Green polynomials $Y^{\lambda}_{\mu}(t)$ in terms of the matrix coefficient:
\begin{align*}
Y^{\lambda}_{\mu}(t)=\langle G_{\lambda}(x;t), p_{\mu}(x) \rangle.
\end{align*}
Therefore, 
$Y^{\lambda}_{\mu}(t)=\langle G_{\lambda}.\mathbbm{1}, p_{\mu} \rangle.$ To compute $Y^{\lambda}_{\mu}(t)$, we introduce the vertex operators $P(z)$ and $P^*(z)$ on the space $\Gamma$ as the following  maps: $\Gamma\longrightarrow \Gamma[[z, z^{-1}]]$ defined by
\begin{align}
P(z)&=\sum_{n\geq 1,odd}p_{n}z^n \\
P^*(z)&=\sum_{n\geq 1,odd}p^*_{n}z^{-n}
\end{align}

By the techniques of vertex operators developed in section 2, we have the following by \eqref{e:Qhallop} and \eqref{e:Qhallop*}:
\begin{align}
G^*(z)P(w)&=\sum_{n\geq 1, odd}(\frac{w}{z})^{n}G^*(z)+P(w)G^*(z) \\
P^*(w)G(z)&=\sum_{n\geq 1,odd}(\frac{z}{w})^{n}G(z)+G(z)P^*(w)
\end{align}

Therefore, their components satisfy the following commutation relations:
\begin{align}\label{e:G*mpn}
G^*_mp_n&=G^*_{m-n}+p_nG^*_m \\
p^*_mG_n&=G_{n-m}+G_np^*_m
\end{align}

Denote by $\nu\lhd\lambda$ if $\nu$ is a subpartition of $\lambda=(\lambda_1,\lambda_2,\cdots,\lambda_l)$, i.e., $\nu=(\lambda_{i_1},\lambda_{i_2},\cdots,\lambda_{i_k})$ for some $1\leq i_1<i_2<\cdots <i_k\leq l.$ Note that $\nu$ could be $\emptyset$ or $\lambda.$ Let $\{\lambda\}_i$ be the set $\{\tau\lhd \lambda\mid|\tau|=i\}$ and $D^{(i)}(\lambda)={\rm Card}\{\lambda\}_i.$ Define $D_t(\lambda)=\sum_{i\geq0}D^{(i)}(\lambda)t^i.$ The following results are useful in our subsequent discussion.
\begin{lem}\label{t:twoidentities}
\cite{JL, JL2} Let $\lambda$ be any partition, then
\begin{align}
D_t(\lambda)&=\prod_{i\geq1}[2]^{m_i(\lambda)}_{t^i}\\
\sum\limits_{\rho \in\mathcal{OP}_n}\frac{(-2)^{l(\rho)}}{z_{\rho}(t)}&=
\begin{cases}
2(t-1)(n)_t&\text{if $n\geq 1$},\\
1&\text{if $n=0$}.
\end{cases}
\end{align}
\end{lem}

We collect some properties of $X^{\lambda}_{\mu}(t)$ as follows.
\begin{prop}\label{t:Green}\cite{JL, M}
Suppose $\lambda, \mu\in\mathcal{P}_n.$ Then the Green polynomial $X^{\lambda}_{\mu}(t)$ satisfies the following properties:\vspace{0.25cm}\\
(1) $X^{\lambda}_{\mu}(t)=\sum_{\nu\in\mathcal{P}_n}K_{\nu\lambda}(t)X^{\nu}_{\mu}(0).$\\
(2) $X^{\lambda}_{\mu}(t)$ is monic of degree $n(\mu).$\\
(3) $X^{\lambda}_{\mu}(0)=\chi^{\lambda}_{\mu}.$\\
(4) $X^{(n)}_{\mu}(t)=1.$\\
(5) $X^{(k, n-k)}_{\mu}(t)=(t-1)[D_t(\mu)t^{-k-1}]_++D^{(n-k)}(\mu).$\\
(6) $X^{\la}_{\mu}(t)=\sum\limits_{i=0}^{n-\la_{1}}\sum\limits_{\tau\in\{\mu\}_i}\sum\limits_{\rho\vdash(n-\la_{1}-i)}\frac{(-1)^{l(\rho)}}{z_{\rho}(t)}X^{\la^{[1]}}_{\tau\cup \rho}(t),$\\
where $[f(t)]_{+}$ is the regular part of the function $f(t)$ in $t$ and $K_{\nu\lambda}(t)$ is the usual $t$-Kostka polynomials \cite{M}.
\end{prop}

\begin{thm}
For partition $\lambda\in\mathcal{SP}_n,$ $\mu\in\mathcal{OP}_n,$ and non-negative integer $k$,
\begin{align}\label{e:G*p}
G_{k}^{*}p_{\mu}&=\sum\limits_{\nu\vartriangleleft \mu} p_\nu G^*_{k+|\nu|-n}=\sum_{i=0}^{n-k} \sum_{\nu\in \{\mu\}_i}p_\nu G^*_{k+i-n} \\ \label{e:p*G}
p^*_{k}G_{\lambda}&=\sum_{i=1}^{l(\lambda)}G_{\lambda_1}\cdots G_{\lambda_i-k}\cdots G_{\lambda_l}.
\end{align}
\end{thm}
\begin{proof}
For \eqref{e:G*p}, we use induction on $l(\mu)$. The first step is immediate. The inductive hypothesis implies that \eqref{e:G*p} holds for $<l(\mu)$. By relation \eqref{e:G*mpn}, we have
\begin{align*}
G^*_{k}p_{\mu}&=G^*_{k-\mu_{1}}p_{\mu_2} \cdots p_{\mu_l}+p_{\mu_1}G^*_{k}p_{\mu_2} \cdots p_{\mu_l} \\
&=\sum\limits_{\nu\vartriangleleft \mu^{\hat{1}}} p_\nu G^*_{k+|\nu|-n}+p_{\mu_1}\sum\limits_{\nu\vartriangleleft \mu^{\hat{1}}} p_\nu G^*_{k+\mu_{1}+|\nu|-n} \\
&=\sum\limits_{\nu\vartriangleleft \mu} p_\nu G^*_{k+|\nu|-n}.
\end{align*}
The argument for \eqref{e:p*G} is similar.
\end{proof}

\begin{thm}\label{t:Siterative}
Let $\lambda\in\mathcal{SP}_n,$ $\mu\in\mathcal{OP}_n,$ then we have the following iterative formula for $Y^{\lambda}_{\mu}(t)$
\begin{align}\label{e:Siterative}
Y^{\lambda}_{\mu}(t)=\sum_{i=0}^{n-\lambda_1}\sum_{\nu\in \{\mu\}_i}\sum_{\rho\in\mathcal{OP}_{n-\lambda_1-i}}\frac{(-2)^{l(\rho)}}{z_{\rho}(t)}
Y_{\nu\cup\rho}^{\lambda^{\hat{1}}}(t).
\end{align}
\end{thm}
\begin{proof}
Note that $(n\geq0)$
\begin{align*}
G^*_{-n}.\mathbbm{1}=\sum_{\rho\in\mathcal{OP}_n}\frac{(-2)^{l(\rho)}}{z_{\rho}(t)}p_{\rho}.
\end{align*}
Then \eqref{e:Siterative} follows from \eqref{e:G*p}.
\end{proof}

Now we want to discuss the properties of $Y^{\lambda}_{\mu}(t)$ similar to those of Green's polynomial $X_{\mu}^{\lambda}(t).$ As an immediate result of \eqref{e:Siterative}, we have $Y^{(n)}_{\mu}(t)=1$ for any $\mu\in\mathcal{OP}_n.$ Furthermore,
\begin{align*}
Y^{(k,n-k)}_{\mu}(t)&=\sum_{i=0}^{n-k}\sum_{\nu\in \{\mu\}_i}\sum_{\rho\in\mathcal{OP}_{n-k-i}}\frac{(-2)^{l(\rho)}}{z_{\rho}(t)} Y_{\nu\cup\rho}^{(n-k)}(t)\\
&=\sum_{i=0}^{n-k-1}\sum_{\nu\in \{\mu\}_i}\sum_{\rho\in\mathcal{OP}_{n-k-i}}\frac{(-2)^{l(\rho)}}{z_{\rho}(t)}+\sum_{\nu\in \{\mu\}_{n-k}}1.
\end{align*}
By Lemma \ref{t:twoidentities}, we have
\begin{align*}
Y^{(k,n-k)}_{\mu}(t)&=\frac{2(t-1)}{t+1}([D_{t^{-1}}(\mu)t^{n-k}]_+-([D_{t^{-1}}(\mu)t^{n-k}]_+\mid_{t=-1}))\\
& \qquad\quad+D^{(n-k)}(\mu)\\
&=\frac{2(t-1)}{t+1}([D_{t}(\mu)t^{-k}]_+-([D_{t}(\mu)t^{-k}]_+\mid_{t=-1}))+D^{(n-k)}(\mu).
\end{align*}

\begin{exmp}
For $\lambda=(4,3),$ $\mu=(5,1,1).$ we have
\begin{align*}
&Y^{(4,3)}_{(5,1,1)}(t)=\frac{2(t-1)}{t+1}([(t+1)^2(t^5+1)t^{-4}]_+\\
&\qquad -([(t+1)^2(t^5+1)t^{-4}]_+\mid_{t=-1}))+D^{(3)}(\mu)\\
&=\frac{2(t-1)}{t+1}(t^3+2t^2+t+0)+0=2t(t^2-1).
\end{align*}
\end{exmp}

It follows from the definition of $Y^{\lambda}_{\mu}(t)$ that ($t=0$)
\begin{align*}
Q_{\lambda}(x)=\sum_{\mu\in\mathcal{OP}_n}z_{\mu}^{-1}2^{l(\mu)}Y^{\lambda}_{\mu}(0)p_{\mu}(x).
\end{align*}
Recall that
\begin{align*}
G_{\lambda}(x;t)=\sum_{\nu\in\mathcal{SP}_n}L_{\nu\lambda}(t)Q_{\nu}(x).
\end{align*}
Combining with the two identities above, we have
\begin{align*}
G_{\lambda}(x;t)=\sum_{\mu\in\mathcal{OP}_n}\sum_{\nu\in\mathcal{SP}_n}z_{\mu}^{-1}2^{l(\mu)}L_{\nu\lambda}(t)Y^{\nu}_{\mu}(0)p_{\mu}(x).
\end{align*}

Since $p_{\mu}(x),$ $\mu\in\mathcal{OP}_n,$ form a basis of $\Gamma_n,$
\begin{align}\label{e:SL}
Y^{\lambda}_{\mu}(t)=\sum_{\nu\in\mathcal{SP}_n}L_{\nu\lambda}(t)Y^{\nu}_{\mu}(0).
\end{align}
Note that the degree of $L_{\nu\lambda}(t)$ is $n(\lambda)-n(\nu),$ it follows that the dominant term on the right-hand of \eqref{e:SL} is $L_{(n)\lambda}Y^{(n)}_{\mu}(0)=2^{l(\lambda)-1}t^{n(\lambda)}$, so the degree of $Y^{\lambda}_{\mu}(t)$ is $n(\mu)$.

Summarizing the above, we have the following parallel properties for $Y^{\lambda}_{\mu}(t)$.

\begin{thm}\label{t:spinG}
The spin Green polynomial $Y^{\lambda}_{\mu}(t)$ for $\lambda\in\mathcal{SP}_n,$ $\mu\in\mathcal{OP}_n$ have the following properties:\vspace{0.25cm}\\
(1) $Y^{\lambda}_{\mu}(t)=\sum_{\nu\in\mathcal{OP}_n}L_{\nu\lambda}(t)Y^{\nu}_{\mu}(0).$\\
(2) $Y^{\lambda}_{\mu}(t)$ is of degree $n(\mu)$ with the leading coefficient $2^{l(\lambda)-1}$.\\
(3) $Y^{\lambda}_{\mu}(0)=2^{\frac{l(\lambda)-l(\mu)+\epsilon(\lambda)}{2}}\zeta^{\lambda}_{\mu}.$\\
(4) $Y^{(n)}_{\mu}(t)=1.$\\
(5) $Y^{(k, n-k)}_{\mu}(t)=\frac{2(t-1)}{t+1}([D_t(\mu)t^{-k}]_+-([D_t(\mu)t^{-k}]_+\mid_{t=-1}))+D^{(n-k)}(\mu).$\\
(6) $Y^{\lambda}_{\mu}(t)=\sum\limits_{i=0}^{n-\lambda_1}\sum\limits_{\nu\in \{\mu\}_i}\sum\limits_{\rho\in\mathcal{OP}_{n-\lambda_1-i}}\frac{(-2)^{l(\rho)}}{z_{\rho}(t)}Y_{\nu\cup\rho}^{\lambda^{\hat{1}}}(t)$.
\end{thm}

\begin{rem}
Morris and Abdel-Aziz \cite{MA} gave a two-row formula and an iterative formula for $\zeta^{\lambda}_{\mu},$ which can be viewed as a
 specialization of (5) at $t=0$ (up to powers of $2$).
\end{rem}

In the following we give tables of spin Green's polynomials for $n\leq 7$. We notice that
$q^{n(\mu)}Y^{\lambda}_{(1^n)}(q^{-1})$ display similarly positivity as the Green polynomials $q^{n(\mu)}X^{\lambda}_{(1^n)}(q^{-1})$, which are
the Poincar\'e polynomials of certain flag varieties \cite{HS}.

\vskip 0.5 in
{\bf Tables for $Y^{\lambda}_{\mu}(t)$}, $n=3, 4, 5, 6, 7$.

\begin{table}[H]
\centering
	
\caption{\label{tab:3}n=3}
	
\begin{tabular}{|c|c|c|}
		
\hline
		
\tabincell{c}{$\mu\backslash \lambda$ } & $(3)$ & $(2,1)$ \\
		
\hline
		
$(3)$ & \tabincell{c}{$1$} & \tabincell{c}{$2(t-1)$}  \\
		
\hline
		
$(1^3)$ & \tabincell{c}{$1$}    & \tabincell{c}{$2t+1$}      \\
		
\hline
		
\end{tabular}
	
\end{table}

\begin{table}[H]
	
\centering
	
\caption{\label{tab:4}n=4}
	
\begin{tabular}{|c|c|c|}
		
\hline
		
\tabincell{c}{$\mu\backslash \lambda$ } & $(4)$ & $(3,1)$  \\
		
\hline
		
$(3,1)$ & \tabincell{c}{$1$} & \tabincell{c}{$2t-1$}  \\
		
\hline

$(1^4)$   & \tabincell{c}{$1$}  & \tabincell{c}{$2(t+1)$}    \\
		
\hline
		
\end{tabular}
	
\end{table}

\begin{table}[H]
	
\centering
	
\caption{\label{tab:5}n=5}
	
\begin{tabular}{|c|c|c|c|}
		
\hline
		
\tabincell{c}{$\mu\backslash \lambda$ } & $(5)$ & $(4,1)$ & $(3,2)$  \\
		
\hline
		
$(5)$ & \tabincell{c}{$1$} & \tabincell{c}{$2(t-1)$} & \tabincell{c}{$2(t-1)^2$} \\
		
\hline

$(3,1^2)$ & \tabincell{c}{$1$}  & \tabincell{c}{$2t$}  & \tabincell{c}{$2t^2-1$}  \\
		
\hline

$(1^5)$   & \tabincell{c}{$1$}    & \tabincell{c}{$2t+3$}    & \tabincell{c}{$2(t^2+3t+1)$}     \\
		
\hline
		
\end{tabular}
	
\end{table}

\begin{table}[H]
	
\centering
	
\caption{\label{tab:6}n=6}
	
\begin{tabular}{|c|c|c|c|c|}
		
\hline
		
\tabincell{c}{$\mu\backslash \lambda$ } & $(6)$ & $(5,1)$ & $(4,2)$ & $(3,2,1)$ \\
		
\hline

$(5,1)$ & \tabincell{c}{$1$} & \tabincell{c}{$2t-1$} & \tabincell{c}{$2t(t-1)$} & \tabincell{c}{$2(t-1)^2(2t^2+2t+1)$}  \\
		
\hline
		
$(3,3)$ & \tabincell{c}{$1$}  & \tabincell{c}{$2(t-1)$}  & \tabincell{c}{$2(t-1)^2$}  & \tabincell{c}{$4(t-1)(t^3-t^2+1)$}  \\
		
\hline

$(3,1^3)$   & \tabincell{c}{$1$}    & \tabincell{c}{$2t+1$}    & \tabincell{c}{$2t^2+2t-1$}    & \tabincell{c}{$4t^4+4t^3-2t^2-2t-1$}   \\

\hline

$(1^6)$   & \tabincell{c}{$1$}    & \tabincell{c}{$2(t+2)$}    & \tabincell{c}{$2t^2+8t+5$}    & \tabincell{c}{$2(2t^4+8t^3+14t^2+5t+1)$}   \\
		
\hline

\end{tabular}
	
\end{table}

\begin{table}[H]
	
\centering
	
\caption{\label{tab:7}n=7}
	
\begin{tabular}{|c|c|c|c|c|c|}
		
\hline
		
\tabincell{c}{$\mu\backslash \lambda$ } & $(7)$ & $(6,1)$ & $(5,2,)$ & $(4,3)$ & $(4,2,1)$\\
		
\hline
		
$(7)$ & \tabincell{c}{$1$} & \tabincell{c}{$2(t-1)$} & \tabincell{c}{$2(t-1)^2$} & \tabincell{c}{$2(t-1)(t^2-t+1)$} & \tabincell{c}{$4t^2(t-1)^2$} \\
		
\hline

$(5,1^2)$ & \tabincell{c}{$1$}  & \tabincell{c}{$2t$}  & \tabincell{c}{$2t^2-1$}  & \tabincell{c}{$2t(t^2-1)$} & \tabincell{c}{$2(t-1)(2t^3+2t^2-1)$} \\

\hline
		
$(3^2,1)$ & \tabincell{c}{$1$}  & \tabincell{c}{$2t-1$}  & \tabincell{c}{$2t(t-1)$}  & \tabincell{c}{$2(t^3-t^2+1)$} & \tabincell{c}{$2(t-1)(2t^3-t+1)$} \\
		
\hline
		
$(3,1^4)$   & \tabincell{c}{$1$}    & \tabincell{c}{$2t+2$}  & \tabincell{c}{$2t(t+2)$}  & \tabincell{c}{$2t^3+4t-1$}    & \tabincell{c}{$2(t+1)(2t^3+2t^2-1)$}   \\
		
\hline
		
$(1^7)$   & \tabincell{c}{$1$}    & \tabincell{c}{$2t+5$}    & \tabincell{c}{$2t^2+10t+9$}    & \tabincell{c}{$2t^3+10t^2+18t+5$} & \tabincell{c}{$4t^4+20t^3+46t^2+28t+7$}  \\
		
\hline
		
\end{tabular}
	
\end{table}

\bigskip

\noindent {\bf Data availability}: All associated data are included in the paper.

\medskip

\vskip30pt \centerline{\bf Acknowledgments}
The project is partially supported by
Simons Foundation grant No. 523868 and NSFC grant No. 12171303.
\bigskip

\bibliographystyle{plain}

\end{document}